\documentclass[11pt]{amsart}
\usepackage{amsmath}
\usepackage{amssymb}
\usepackage{amsthm}
\usepackage{latexsym}
\usepackage{pstricks}
\usepackage{amscd}

\date{}

\def\opn#1#2{\def#1{\operatorname{#2}}} 
\opn\chara{char} \opn\length{\ell} \opn\pd{pd} \opn\rk{rk}
\opn\projdim{proj\,dim} \opn\injdim{inj\,dim} \opn\rank{rank}
\opn\depth{depth} \opn\codepth{codepth} \opn\grade{grade}
\opn\height{height} \opn\embdim{emb\,dim} \opn\codim{codim}

\opn\Tr{Tr} \opn\bigrank{big\,rank}
\opn\superheight{superheight}\opn\lcm{lcm}
\opn\trdeg{tr\,deg}%
\opn\reg{reg} \opn\lreg{lreg} \opn\skel{skel} \opn\Gr{Gr}
\opn\dim{dim} \opn\arithdeg{arithdeg}

\opn\diam{diam}

\opn\div{div} \opn\Div{Div} \opn\cl{cl} \opn\Cl{Cl}
\opn\Spec{Spec} \opn\Supp{Supp} \opn\supp{supp} \opn\Sing{Sing}
\opn\Ass{Ass}
\opn\Ann{Ann} \opn\Rad{Rad} \opn\Soc{Soc}
\opn\Sym{Sym} \opn\Ker{Ker} \opn\Coker{Coker} \opn\Im{Im}
\opn\Hom{Hom} \opn\Tor{Tor} \opn\Ext{Ext} \opn\End{End}
\opn\Aut{Aut} \opn\id{id} \opn\ini{in} \opn\tr{tr}

\opn\nat{nat}\opn\it{it}
\opn\pff{proof}
\opn\Pf{proof} \opn\GL{GL} \opn\SL{SL} \opn\mod{mod} \opn\ord{ord}
\opn\aff{aff} \opn\con{conv} \opn\relint{relint} \opn\st{st}
\opn\lk{lk} \opn\cn{cn} \opn\core{core} \opn\vol{vol}
\opn\link{link} \opn\star{star} \opn\skel{skel}
\opn\cone{cone} \opn\star{star} \opn\skel{skel}

\opn\gr{gr}

\def\pot#1#2{#1[\kern-0.28ex[#2]\kern-0.28ex]}

\opn\dirlim{\underrightarrow{\lim}}
\opn\inivlim{\underleftarrow{\lim}}
\let\union=\cup

\let\ol=\overline
\let\wt=\widetilde

\def\Implies{\ifmmode\Longrightarrow \else
     \unskip${}\Longrightarrow{}$\ignorespaces\fi}
\def\implies{\ifmmode\Rightarrow \else
     \unskip${}\Rightarrow{}$\ignorespaces\fi}
\def\iff{\ifmmode\Longleftrightarrow \else
     \unskip${}\Longleftrightarrow{}$\ignorespaces\fi}

\let\:=\colon
\newtheorem{Theorem}{Theorem}[section]
\newtheorem{Lemma}[Theorem]{Lemma}
\newtheorem{Corollary}[Theorem]{Corollary}
\newtheorem{Proposition}[Theorem]{Proposition}

\newtheorem{Example}[Theorem]{Example}

\let\epsilon\varepsilon
\let\phi=\varphi
\let\kappa=\varkappa
\def\qed{\ifhmode\textqed\fi
   \ifmmode\ifinner\quad\qedsymbol\else\dispqed\fi\fi}
\def\textqed{\unskip\nobreak\penalty50
    \hskip2em\hbox{}\nobreak\hfil\qedsymbol
    \parfillskip=0pt \finalhyphendemerits=0}
\def\dispqed{\rlap{\qquad\qedsymbol}}

\opn\Gin{Gin}

\def\GG{{\mathcal G}}
\def\BB{{\mathcal B}}
\def\A{{\mathcal A}}

\def\FF{{\mathcal F}}
\def\MM{{\mathcal M}}

\opn\inii{in} \opn\inim{inm} \opn\rate{rate}

\numberwithin{equation}{section}

\begin{document}
\title{Construction of Cohen-Macaulay binomial edge ideals}
\author[Asia Rauf]{Asia Rauf}

\author[Giancarlo Rinaldo]{Giancarlo Rinaldo}
\address{Asia Rauf, The Abdus Salam International Centre for Theoretical Physics (ICTP), Trieste-Italy.} \email{arauf@ictp.it, asia.rauf@gmail.com}
\address{Giancarlo Rinaldo, Dipartimento di Matematica,
Universit\`a di Messina,
Viale Ferdinando Stagno d'Alcontres, 31
98166 Messina, Italy.}
\email{giancarlo.rinaldo@tiscali.it}

\subjclass[2000]{Primary 13F55, Secondary 13H10}
\date{\today}
\keywords{Depth, Binomial ideal, Cohen--Macaulay}

\begin{abstract}
We discuss  algebraic and homological properties  of binomial edge ideals associated to  graphs which are obtained by gluing  of subgraphs and the formation of cones.
\end{abstract}

\maketitle

\section*{Introduction}

In this paper we study binomial edge ideals associated with  finite simple graphs. There are several natural ways to associate ideals with graphs and study their algebraic properties in terms of the underlying graphs. Let $K$ be a field. Classically one associates with  a graph $G$ on  the vertex set $[n]$ its edge ideal $I_G$ in the polynomial ring $K[x_1,\ldots,x_n]$ over  $K$. The ideal $I_G$   is generated by the monomials $x_ix_j$,  where $\{i, j\}$ is an edge of $G$.  Edge ideals of a graph have been introduced by Villarreal \cite{Vi2} in 1990, where he studied the Cohen-Macaulay property of such ideals. Many authors have focused their attention on such ideals (see for example \cite{SVV}, \cite{HH1}). 

In 2010, binomial edge ideals were introduced in \cite{HH} and appeared independently, but at the same time, also in \cite{MO}. Let $S = K[x_1,\cdots, x_n, y_1,\cdots, y_n]$ be the polynomial ring in $2n$ variables with coefficients in a field $K$. Let $G$ be a graph on vertex set $[n]$. For each edge $\{i,j\}$ of $G$ with $i < j$, we associate a binomial $f_{ij} = x_iy_j - x_jy_i$. The ideal $J_G$ of $S$ generated by $f_{ij} = x_iy_j - x_jy_i$ such that $i<j$ , is called \textit{the binomial edge ideal} of $G$.

Observe that any ideal generated  by a set of $2$-minors  of an  $2\times n$-matrix $X$ of indeterminates may be viewed as the binomial edge ideal of a graph. For example, the ideal of $2$-minors of $X$  is the binomial edge ideal of the complete graph on $[n]$.
The binomial edge ideal of a line graph is another example of binomial edge ideal. It is the ideal of all adjacent minors of $X$. This example appears the first time  in \cite{DES}. Algebraic properties of binomial edge ideals in terms of properties of the underlying graph were studied in \cite{HH}, \cite{CR} and \cite{HEH}. In \cite{HEH}, the authors considered some special classes of graphs and studied the Cohen-Macaulay property of these graphs.  However, the classification of Cohen-Macaulay binomial edge ideals in terms of the underlying graphs is still widely open, and it seems rather hopeless to give a full classification.

 Let $G$ be a graph with $G=G_1\cup G_2$ where $G_1$ and $G_2$ are two subgraphs of $G$.  In this paper we study the question under which conditions   $J_G$ is unmixed or Cohen-Macaulay, provided $J_{G_1}$ and $J_{G_2}$ have this property.

 Suppose $V(G_1)\cap V(G_2)=\{v\}$, where $v$ is free vertex of $G_1$ and $G_2$ (in the corresponding clique complexes). In this situation it is shown in Proposition \ref{prop:unmixedfree} and Theorem \ref{the:union} that $J_G$ is unmixed (Cohen-Macaulay) if and only if $J_{G_1}$ and $J_{G_2}$ are unmixed (Cohen-Macaulay).

 Now let $G=G_1\cup\ldots\cup G_r$, where the $G_i$ are subgraphs of $G$ with the property that
 \begin{enumerate}
 \item[(1)] $|V(G_i)\cap V(G_j)|\leq 1$ and $V(G_i)\cap V(G_j)\cap V(G_k)=\emptyset$ for $i\neq j\neq k$,

 \item[(2)] any  $v\in V(G_i)\cap V(G_j)$ with $i\neq j$ is a free vertex in the corresponding clique complex of $G_i$ and $G_j$.
 \end{enumerate}
 We associate with this decomposition of $G$ the graph $G^f$ with   vertex set   $\{1,\ldots, r\}$ and  edge set
 $\{\{i,j\}:V(G_i)\cap V(G_j)\neq\emptyset\}$.  As an application of Theorem \ref{the:union}, we show in Corollary \ref{cor:fvi} that $J_G$ is Cohen-Macaulay if and only if $J_{G_i}$ is Cohen-Macaulay for all $i$, provided $G^f$ is a tree.

 Let $H$ be a graph and let $G=\cone(v,H)$ be the cone on $H$. In Theorem~\ref{the:cone1conn}, it is shown that if $H$ is connected graph and $J_H$ is unmixed,  then $H$ is complete graph if and only if $J_G$ is unmixed. It is clear that $J_G$ is Cohen-Macaulay if $H$ is a complete graph.  Without assuming that $J_H$ is unmixed, we conjecture that $H$ is complete graph if $J_G$ is Cohen-Macaulay. By computer experiments, using CoCoA (see \cite{Co})  and Nauty (see \cite{Mk}), we verified this conjecture for all graphs with at most $9$ vertices.
 Now assume that $H$ has  $r\geq 2$ connected components. We show in  Lemma \ref{cor:only2} and Corollary \ref{cor:unmixed} that  $J_{G}$ is unmixed  if and only if $J_H$ is unmixed and $r=2$.  Moreover, we show in Theorem \ref{theo:cone2CM} that $J_G$ is Cohen-Macaulay if $J_{H}$ is Cohen-Macaulay. It is an open question whether the converse of this statement is true.

\section{Preliminaries}\label{sec:pre}
In this section we recall some concepts and a notation on graphs and on simplicial
complexes that we will use in the article.

Let $G$ be a simple graph with vertex set $V(G)$ and the edge set $E(G)$. A subset $C$ of $V(G)$ is called a \textit{clique} of $G$ if for all $i$ and $j$ belonging to $C$ with $i \neq j$ one has $\{i, j\} \in E(G)$. A vertex of a graph is called cut point if the removal of the vertex increases the number of connected components. Let $v\not\in V(G)$. The \textit{cone} of $v$ on $G$, namely $\cone(v,G)$, is the graph with vertices $V(G)\cup \{v\}$ and edges $E(G)\cup \{\{u,v\}:u\in V(G)\}$.

Let $G_1$ and $G_2$ be graphs. We set $G=G_1\cup G_2$  (resp. $G=G_1\sqcup G_2$ where $\sqcup$ is disjoint union) where $G$ is the graph with $V(G)=V(G_1)\cup V(G_2)$  (resp. $V(G)=V(G_1)\sqcup V(G_2)$) and $E(G)=E(G_1)\cup E(G_2)$ (resp. $E(G)=E(G_1)\sqcup E(G_2)$).

Set $V = \{x_1, \ldots, x_n\}$. A \textit{simplicial complex}
$\Delta$ on the vertex set $V$ is a collection of subsets of $V$
such that
\begin{enumerate}
\item[(i)] $\{x_i\} \in \Delta$  for all $x_i \in V$;
\item[(ii)] $F \in \Delta$ and $G\subseteq F$ imply $G \in \Delta$.
\end{enumerate}
An element $F \in \Delta$ is called a \textit{face} of $\Delta$.

A maximal face of $\Delta$  with respect to inclusion is called a \textit{facet} of $\Delta$.
A vertex $i$ of $\Delta$ is called a free vertex of $\Delta$ if $i$ belongs to exactly one facet.

If $\Delta$ is a simplicial complex  with facets $F_1, \ldots, F_q$, we call $\{F_1, \ldots, F_q\}$ the facet set of $\Delta$ and we denote it by $\FF(\Delta)$.

The \textit{clique complex} $\Delta(G)$ of $G$ is the simplicial complex whose facets are the cliques of $G$. Hence a vertex $v$ of a graph $G$ is called \textit{free vertex} if it belongs to only one clique of $\Delta(G)$.

We need notations and results  from \cite{HH} (section 3) that we recall for sake of completeness.

Let $T\subseteq [n]$, and let $\ol{T}=[n]\setminus T$. Let $G_1,\ldots,G_{c(T)}$ be the connected components of the induced subgraph on $\ol{T}$, namely $G_{\ol{T}}$. For each $G_i$, denote by $\wt{G}_i$ the complete graph on the vertex set $V(G_i)$. We set 
\begin{equation}\label{eq:prime}
P_T(G)=(\bigcup_{i\in T}\{x_i,y_i\},J_{\wt{G}_1},\ldots,J_{\wt{G}_{c(T)}} ), 
\end{equation}
that is a prime ideal. Then $J_G=\bigcap_{T\subset [n]}P_T(G)$. If there is no confusion possible, we write simply $P_T$ instead of $P_T(G).$ Moreover, $\height P_T=n+|T|-c(T)$ (see \cite[Lemma 3.1]{HH}).
 We denote by $\MM(G)$ the set of minimal prime ideals of $J_G$.

If each $i\in T$ is a cut point of the graph $G_{\ol{T}\cup \{i\}}$, then we say that $T$ has {\em cut point property for $G$.} We denote by $\mathcal{C}(G)$ the set of all $T\subset V(G)$ such that $T$ has cut point property for $G$.
\begin{Lemma}\label{lem:cutpoint}\cite{HH}
 $P_T(G)\in \MM(G)$ if and only if $T\in\mathcal{C}(G)$.
\end{Lemma}
\section{Gluing of graphs and binomial edge ideals}
In this section we study unmixed and Cohen-Macaulay properties of the binomial edge ideal of a graph which is constructed by gluing two graphs with a free vertex that belongs to both graphs.

Since a binomial edge ideal $J_G$ is Cohen-Macaulay (resp. unmixed) if and only if  $J_H$ is Cohen-Macaulay (resp. unmixed) for each connected component $H$ of $G$,  we assume from now on that the graph $G$ is connected unless otherwise stated.

We begin by the following
\begin{Proposition}\label{prop:free0}
 Let $G$ be a graph,  $\Delta(G)$ its clique complex and $v\in V(G)$. The following conditions are equivalent:
\begin{enumerate}
 \item[(a)] There exists $T\in\mathcal{C}(G)$ such that $v\in T$;
 \item[(b)] $v$ is not a free vertex of $\Delta(G)$.
\end{enumerate}
\end{Proposition}
\begin{proof}
 (a) \implies (b): Let $F_i$ be a facet of $\Delta(G)$ such that $v\in F_i.$ We want to show that $v$ is not a free vertex of $\Delta(G)$. Suppose on contrary that $v\notin F_j\in\Delta(G)$ for all facets $F_j\neq F_i$. Consider the graph $G_{\ol{T}\cup\{v\}}$. We claim that $v$ cannot be a cut point for the induced subgraph $G_{\ol{T}\cup\{v\}}$. In fact for any pair of edges $\{u,v\}$, $\{v,w\}\in E(G_{\ol{T}\cup\{v\}})$, we get $\{u,v\}$, $\{v,w\}\in F_i$. Therefore $\{u,w\}\in  E(G_{\ol{T}\cup\{v\}})$.

 (b) \implies (a): Let $F_i, F_j\in\FF(\Delta(G))$ such that $i\neq j$ and $v\in F_i\cap F_j$. We will show that there exists $T\in\mathcal{C}(G)$ with $v\in T$. Let $T'=F_i\cap F_j$ and let $F_i'=F_i\setminus F_j$ and $F_j'=F_j\setminus F_i$.
 Let $\mathcal C'=\{T\subset V(G)\mid T'\subset T,F_i' \text{ and }F_j' \text{ belong to different components of } G_{\ol{T}} \}.$ Then $\mathcal C'\neq\emptyset$, since $T=[n]\setminus (F_i'\cup F_j')\in \mathcal C'.$ Consider $T\in \mathcal C'$ such that $T$ is minimal with respect to inclusion of sets. We claim that $T$ has cutpoint property. For all $u\in T'$ this property is satisfied, since $u\in F_i\cap F_j.$ Suppose there exists a $u\in T\setminus T'$ such that $u$ is not cut point. Consider the set $T\setminus \{u\}$. We see that $T'\subset (T\setminus \{u\})$ and also $F_i'$ and $F_j'$ belong to different components of the induced graph $G_{\ol{T} \cup \{u\}}$, since $F_i'$ and $F_j'$ belong to different components of the graph $G_{\ol{T}}$. This implies that $T\setminus \{u\}\in\mathcal C'$, a contradiction.

\end{proof}
$H$ subgraph of $G$.
We use the following known fact in the proofs of Lemma \ref{lem:cutpoint pro} and Lemma \ref{lem:unfreev}. $v$ is a cut point of a graph $G$ if and only if there exist $u,w\in V(G)$ such that $v$ is in every path connecting $u$ and $w$ (see \cite[Theorem 3.1]{Ha}).
\begin{Lemma}\label{lem:cutpoint pro}
Let $G$ be a graph with $v\in V(G)$ such that $v$ is a free vertex in $\Delta(G)$, and let $F$ be the facet of $\Delta(G)$ with $v\in F$. Let $T\subset V(G)$ with $F\setminus\{v\}\not\subseteq T$. Then the following conditions are equivalent:
\begin{enumerate}
\item[(a)] $T\in\mathcal{C}(G)$;
\item[(b)] $v\not\in T$ and $T\in\mathcal C (G\setminus\{v\})$.
\end{enumerate}
\end{Lemma}

\begin{proof} (a) \implies (b): It follows from  Proposition~\ref{prop:free0}(a)\implies(b) that $v\not\in T$. Now let $u\in T$, then $u\neq v$. Let $G'=G\setminus\{v\}$. We want to show that $u$ is a cut point for the graph $H'=G'_{\ol{T}\cup\{u\}}$.

Since $T\in\mathcal C(G)$, it follows that  $u$ is a cut point of $H=G_{\ol{T}\cup\{u\}}$. This implies that there exist two vertices $u_1$, $u_2\in V(H)$ such that $u$ belongs to all the paths from $u_1$ to $u_2$. If $\{u_1,u_2\}\cap \{v\}=\emptyset$ then $u$ is also a cut point of $H'$. Let $u_1=v$. Consider a path
\[
 \pi=v,v',\ldots,u_2.
\]
If $v'\neq u$, then also all the paths from $v'$ to $u_2$ pass through $u$, and we are done.
If $v'=u$,  then since $T\nsupseteq F\setminus \{v\}$ there exists $v''\in F\setminus (T\cup \{v\})$ with $\{v,v'\}\cap \{v''\}=\emptyset$. It follows that all the paths from $v''$ to $u_2$ pass through $u$. This shows that $T\in\mathcal C(G\setminus\{v\})$.

(b)\implies (a): Consider $u\in T$. We claim that $u$ is cut point of $H=G_{\ol{T}\cup\{u\}}$. Suppose on the contrary that $u$ is not a cut point of $H$. Since $u$ is a cut point of $H'=G'_{\ol{T}\cup\{u\}}$, there exist two vertices $u_1$, $u_2\in V(H')$ such that all the paths from $u_1$ to $u_2$ pass through $u$. Moreover there exists a path
\[
 \pi=u_1,\ldots,v,\ldots,u_2
\]
in $H$ that does not pass through $u$. This implies that $\{v',v\},\{v'',v\} \in E(H)$ and
\[
 \pi=u_1,\ldots,v',v,v'',\ldots,u_2.
\]
Since $F$ is a clique, $\{v',v''\}\in E(H)$ and $\{v',v''\}\in E(H')$. Hence we can obtain a new path
\[
 \pi'=u_1,\ldots,v',v'',\ldots,u_2.
\]
that is contained in $H'$ and does not pass through $u$, a contradiction.
\end{proof}
Note that (b) implies (a) even without assuming that $F\setminus\{v\}\not\subseteq T$.
\begin{Lemma}\label{lem:unfreev}
Let $G=G_1\cup G_2$ be a graph such that $V(G_1)\cap V(G_2)=\{v\}$ and $v$ be a free vertex of $\Delta(G_1)$ and $\Delta(G_2)$. Let $v\in F_1\in \Delta(G_1)$, $v\in F_2\in \Delta(G_2)$. Then
\[
\begin{array}{rcl}
 \mathcal C(G) =\A\union\BB,
\end{array}
\]
where
\[
\A= \{T\subset V(G):T=T_1\cup T_2, T_i\in\mathcal C(G_i)\text{ for }i=1,2\},
\]
and
\[
\BB= \{T\subset V(G):T=T_1\cup T_2\cup\{v\}, T_i\in\mathcal C(G_i)\text { and } F_i\nsubseteq T_i\cup \{v\}\text { for } i=1,2\}.
\]
\end{Lemma}
\begin{proof}
We first show that $\A\union\BB\subseteq\mathcal{C}(G)$. Let  $T=T_1\cup T_2$ where $T_i\in\mathcal C (G_i)$ for $i=1,2$. By Proposition \ref{prop:free0}, $v\notin T_1$ and $v\notin T_2$, since $v$ is a free vertex. It follows that $T\in\mathcal C(G)$ if and only if $T_i\in\mathcal C(G_i)$ for $i=1,2$. Therefore $\A\subset \mathcal C(G)$.

Now suppose $T=T_1\cup T_2 \cup\{v\}$ where each $T_i\in\mathcal C(G_i)$ and $F_i\nsubseteq T_i\cup \{v\}$ for $i=1,2$.
Consider the graph $G'=G_{\ol{T}\cup\{v\}}.$ Since $F_i\nsubseteq T_i\cup \{v\}$ for $i=1,2$ there exist $u_1\in F_1\cap G'$ and $u_2\in F_2\cap G'$ and obviously all the paths from $u_1$ to $u_2$ pass through $v$. Hence $v$ is a cut point of $G'$.

Let $u\in T$ with $u\neq v$. Then the graph $G_{\ol{T}\cup\{u\}}=G_1'\sqcup G_2'$ where $G_i'$ is a graph on the subset of $V(G_i)\setminus\{v\}$ for $i=1,2$. We may assume that $u\in T_1$ and in this case we get $G_1'=(G_1)_{\ol{T}_1\cup\{u\}}\setminus\{v\}$ and $G_2'=(G_2)_{\ol{T}_2}\setminus\{v\}.$ It follows that $u$ is a cut point of $G_{\ol{T}\cup\{u\}}$ if and only if $u$ is a cut point of $G_1'$. Since $u\in T_1$ is a cut point of $(G_1)_{\ol{T}_1\cup\{u\}}$, it is also a cut point of $G_1'$ by Lemma \ref{lem:cutpoint pro}.

In order to prove other inclusion, we consider $T\in\mathcal C(G)$. Suppose $v\not\in T$ and also $T=T_1\cup T_2$ where $T_i\subset V(G_i)$ for $i=1,2$. Also in this case it is clear that $T\in\mathcal C(G)$ if and only if $T_i\in\mathcal C(G_i)$ for $i=1,2$.

If $v\in T$, then $T=T_1\cup T_2 \cup\{v\}$ with $T_i=T\cap V(G_i) \setminus \{v\}$ for $i=1,2$. Since $T\in\mathcal C(G)$, $v$ is a cut point for the graph $G'=G_{\ol{T}\cup\{v\}}$. This implies that there exist $u, w\in V(G')$ such that $v$ belong to every path from $u$ to $w$. We have $\{u',v\},\{w',v\}\in E(G')$ such that $u'$ and $w'$ belong to a path connecting $u$ and $w$. We may assume $u'\in F_1$. Then $w'\in F_2$.  In fact if $u'$, $w'\in F_1$ then there exists a path not containing $v$ since $F_1$ is a clique. Therefore $F_i\nsubseteq T_i\cup\{v\}$ for $i=1,2$. Let $u\in T$ with $u\neq v$. We may assume that $u\in T_1$. Since the graph $G_{{\ol{T}}\cup\{u\}}$ has two disjoint subgraphs, one defined on the vertex set $V(G_1)\setminus \{v\}$ and the second one on the vertex set $V(G_2)\setminus \{v\}$,  we focus our attention on the graph  $(G_1)_{\ol{T}_1\cup\{u\}}\setminus\{v\}$.  By Lemma \ref{lem:cutpoint pro}, since $u$ is a cut point of $G_{{\ol{T}}\cup\{u\}}$ it follows that $u$ is a cut point of $(G_1)_{\ol{T}_1\cup\{u\}}\setminus\{v\}$.
\end{proof}
\begin{Corollary}\label{cor:hfv}
Let $G=G_1\cup G_2$ be a graph such that $V(G_1)\cap V(G_2)=\{v\}$ and $v$ be a free vertex of $\Delta(G_1)$ and $\Delta(G_2)$. Then
 \[
 \height P_T(G)=\height P_{T_1}(G_1)+\height P_{T_2}(G_2)
 \]
for all $T\in \mathcal C(G)$, $T_1\in \mathcal C(G_1)$ and $T_2\in \mathcal C(G_2)$ defined as in Lemma \ref{lem:unfreev}.
\end{Corollary}
\begin{proof}
Let $|V(G_1)|=m_1$ and $|V(G_2)|=m_2$. We have $|V(G)|=n=m_1+m_2-1$.
By Lemma \ref{lem:unfreev} there are two cases to consider. Let $T_1\in\mathcal C(G_1)$ and $T_2\in\mathcal C(G_2)$ and use the same notation of Lemma \ref{lem:unfreev}, we have:
 
Case $1.$ If $T\in\A$, that is $T=T_1\cup T_2$,  then $c(T)=c(T_1)+c(T_2)-1$.

 Case $2.$ If $T\in \BB$, that is $T=T_1\cup T_2\cup\{v\}$, then $c(T)=c(T_1)+c(T_2)$.

\noindent
Since $\height P_{T}=n+|T|-c(T)$, we easily obtain the required formula.
\end{proof}
The following property of unmixed binomial edge ideals is observed in \cite{HH} and \cite{HEH}.
 \begin{Lemma}\label{lem:unmcomp}
  Let $G$ be a graph. Then the following conditions are equivalent:
\begin{enumerate}
 \item[(a)] $J_G$ is unmixed;
 \item[(b)] for all $T\in \mathcal C(G)$, we have $c(T)=|T|+1$.
\end{enumerate}

 \end{Lemma}
 \begin{proof}
 (a)\implies (b): Let $T\in\mathcal C(G)$. Since $J_G$ is unmixed and $\emptyset\in\mathcal C(G)$, $n+|T|-c(T)=\height P_T=\height P_\emptyset=n-1$ implies the required assertion.

(b) \implies  (a): We get $\height P_T=n+|T|-c(T)=n-1$ for all $T\in\mathcal C(G)$ by our assumption. Hence $J_G$ is unmixed.
\end{proof}

\begin{Proposition}\label{prop:unmixedfree}
Let $G=G_1\cup G_2$ such that $V(G_1)\cap V(G_2)=\{v\}$ and $v$ be a free vertex in $\Delta(G_1)$ and $\Delta(G_2)$. Then
$J_G$ is unmixed if and only if $J_{G_1}$ and $J_{G_2}$ are unmixed.
\end{Proposition}
\begin{proof}
Let $J_{G_1}$ and $J_{G_2}$ be unmixed. Let $T_1\subset V(G_1)$ and $T_2\subset V(G_2)$ such that $T_1\in \mathcal C(G_1)$, $T_2\in \mathcal C (G_2)$. By Lemma \ref{lem:unmcomp}  we get $c(T_1)=|T_1|+1$ and $c(T_2)=|T_2|+1$. Let $T\subset V(G)$ such that $T\in \mathcal C(G)$. This implies that either $T=T_1\cup T_2$ or $T=T_1\cup T_2\cup\{v\}$ with $F_i\setminus\{v\}\nsubseteq T_i$. If $T=T_1\cup T_2,$ then  we get $c(T)=c(T_1)+c(T_2)-1$ since $c(T_1)$ and $c(T_2)$ both count the connected component containing $v$. If $T=T_1\cup T_2\cup\{v\},$ then we have $c(T)=c(T_1)+c(T_2)$. In fact since $F_i\setminus\{v\}\nsubseteq T_i$ then $c(T_i)=c(T_i\cup\{v\})$ for $i=1,2$. It follows from both cases that $c(T)=|T|+1$.

Now we suppose that $J_G$ is unmixed, that is $c(T)=|T|+1$ for all $T\in \mathcal C(G).$ Consider $T=T_1\cup T_2$ where $T_1\in \mathcal C (G_1)$, $T_2\in \mathcal C (G_2)$. Since $\emptyset\in\mathcal C(G_i)$ for $i=1,2$, we may assume $T=T_1$. We get $c(T)=c(T_1)$ since $V(G_1)\cap V(G_2)=\{v\}$ and $v$ is a free vertex. It follows that $c(T_1)=|T_1|+1$ for all $T_1\in\mathcal C(G_1)$. Hence $J_{G_1}$ is unmixed. By symmetry, we obtain $J_{G_2}$ is unmixed, too.
\end{proof}

In \cite{HH1}, the authors introduce the {\em admissible path} in order to compute Gr\"obner bases of the binomial edge ideals. We will use this notion in the proof of Theorem \ref{the:union}. A path $i=i_0,i_1,\ldots,i_r=j$ in a graph $G$ is called {\em admissible}, if
\begin{enumerate}
 \item $i_k\neq i_\ell$ for $k\neq \ell$;
 \item for each $k=1,\ldots,r-1$ one has either $i_k<i$ or $i_k>j$;
 \item for any proper subset $\{j_1,\ldots,j_s\}$ of $\{i_1,\ldots,i_{r-1}\}$, the sequence $i,j_1,\ldots,j_s,j$ is not a path.
\end{enumerate}

\begin{Theorem}\label{the:union}
Let $G_1$ and $G_2$ be graphs such that $V(G_1)\cap V(G_2)=\{v\}$ and $v$ be a free vertex in $\Delta(G_1)$ and $\Delta(G_2)$, and let $G=G_1\cup G_2$. Then
\begin{equation}\label{eq:depth}
\depth S/J_G=\depth S_1/J_{G_1}+\depth S_2/J_{G_2}-2
\end{equation}
where $S_i=K[\{x_j,y_j:j\in V(G_i)\}]$ for $i=1,2$. In addition, $J_G$ is Cohen-Macaulay if and only if $J_{G_1}$ and $J_{G_2}$ are Cohen-Macaulay.
\end{Theorem}
\begin{proof}
 Let $v'$ be a vertex such that $v'\not\in V(G)$. We define a graph $G'=G_1\sqcup G_2'$, where $G_1$ is the same graph as in the statement and $G_2'$ is the graph with  $V(G_2)=\{V(G_2)\setminus \{v\}\}\cup \{v'\}$ and $E(G_2')=E(G_2\setminus \{v\})\cup \{\{i,v'\}:\{i,v\}\in E(G_2)\}$.
Let $S=K[\{x_i,y_i:i\in V(G)\}]$ and let $S'=K[\{x_i,y_i:i\in V(G')\}]$. Hence
\[
 S'\cong S[x_{v'},y_{v'}].
\]
Let $l_x=x_v -x_{v'}$ and $l_y=y_v -y_{v'}$. Note that
\[
 S'/(J_{G'},l_x,l_y)\cong S/J_G.
\]
To obtain the statement it suffices to prove that the sequence $l_y$, $l_x$ is regular on $S'/J_{G'}$.
Firstly we show that $l_y$ is regular on $S'/J_{G'}$. Since
\[
 J_{G'}: l_y=\bigcap_{T\in \mathcal C(G')} (P_T:l_y),
\]
it is sufficient to verify that
\[
 P_T:l_y=P_T,\hspace{0.5cm}\text{ for all } \,\,\, T\in \mathcal C(G').
\]
We actually show that $l_y\not\in P_T$. Then this implies that $P_T:l_y=P_T$, because $P_T$ is a prime ideal. We have
\[
 P_T=(\bigcup_{i\in T}\{x_i,y_i\},J_{\wt{G'_1}},\ldots,J_{\wt{G'}_{c(T)}} ).
\]
By Proposition \ref{prop:free0} it follows that
\[l_y\notin (\bigcup_{i\in T}\{x_i,y_i\}),\]
since $v$ and $v'$ are free vertices. Since $l_y$ is a linear form it cannot be obtained by linear combination of quadratic elements in the set of generators of $P_T$, hence $l_y\not\in P_T$ as desired.

We claim that $(J_{G'},l_y):l_x=(J_{G'},l_y)$. We may assume that $V(G_1)=\{1,2,\ldots,n\}$,  $V(G_2)=\{n+1,\ldots,m+n\}$, $v=n$ and $v'=n+1$. For the proof of the claim we describe the Gr\"obner basis of $J_{G'}+(l_y)$. We fix a lexicographic ordering induced by the following order of the variables
\begin{equation}
\label{eq:order}
 x_1>x_2>\cdots >x_{n+m}>y_1>y_2\cdots >y_{n+m}.
\end{equation}
Given an admissible path
\[
 \pi:i=i_0,i_1,\ldots,i_r=j
\]
from $i$ to $j$ with $i<j$ we associate the monomial
\[
 u_\pi=(\prod_{i_k>j}x_{i_k})(\prod_{i_\ell<i}y_{i_\ell}).
\]
Then
\begin{equation}\label{eq:gb0}
 \mathcal{G'}=\{ u_\pi f_{ij}:\pi \mbox{ is an admissible path from $i$ to $j$}\}.
\end{equation}
is  a  Gr\"obner bases of $J_{G'}$ (see \cite{HH} and \cite{MO}).

\medskip
We claim that
\begin{multline}
\label{eq:gbly}
 \GG= \{l_y\}\cup \{ u_\pi f_{ij}:\pi \mbox{ is an admissible path from $i$ to $j\neq n$}\}\cup\\
 \cup \{ u_\pi (x_i y_{n+1} -x_n y_i):\pi \mbox{ is an admissible path from $i$ to $j=n$}\}.
\end{multline}
is a Gr\"obner basis of $J_{G'}+(l_y)$.

Let $\GG_0=\GG'\cup \{l_y\}$. By \eqref{eq:gb0} and Buchberger's criterion all the $S$-pairs of polynomials in $\GG'$ reduce to $0$. Hence we only have to consider the $S$-pairs
\[
 S(l_y,u_\pi f_{ij})
\]
for all $u_\pi f_{ij}\in \GG'$. If $y_n$ does not divide $\ini(u_\pi f_{ij})=u_\pi x_i y_j$ the $S$-pair reduces to $0$. If $y_n$ divides $u_\pi x_iy_j$ then $\pi$ is an admissible path of the connected graph $G_1$ and since $n$ is the maximum index, by definition of admissible path, $j=n$. Therefore
\[
 S(l_y,u_\pi f_{in})=-u_\pi (x_i y_{n+1} -x_n y_i),
\]
with $\ini(- u_\pi (x_i y_{n+1} -x_n y_i))= - u_\pi x_i y_{n+1}$.
We want to show that
\begin{equation}\label{eq:gb}
\begin{split}
 \GG_1 = \{l_y\}\cup \{ u_\pi f_{ij}:\pi \mbox{ is an admissible path from $i$ to $j$}\}\cup \\
 \cup \{ u_\pi (x_i y_{n+1} -x_n y_i):\pi \mbox{ is an admissible path from $i$ to $j=n$}\}.
\end{split}
\end{equation}
is a Gr\"obner basis of $J_{G'}+(l_y)$.  
Since $S(l_y,u_\pi f_{ij})$ reduce to $0$ by the binomials described in the third set of \eqref{eq:gb} and $S(u_\pi f_{ij},u_\sigma f_{kl})$ reduce to $0$ by the binomials described in the second set of \eqref{eq:gb} it remains to investigate the $S$-pairs of the form
\begin{enumerate}
 \item $S(u_\pi(x_i y_{n+1} -x_n y_i), u_\sigma(x_j y_{n+1} -x_n y_j))$ and
 \item $S(u_\pi(x_i y_{n+1} -x_n y_i), u_\sigma f_{kl})$.
\end{enumerate}

Case (1): If $i=j$ then the $S$-polynomial itself is $0$. If $i\neq j$, then
\[
 S(u_\pi(x_i y_{n+1} -x_n y_i), u_\sigma(x_j y_{n+1} -x_n y_j))=S(u_\pi f_{in}, u_\sigma f_{jn}),
\]
and the assertion follows since $\GG_1\supset \GG'$.

Case (2): If $\{k,l\}\cap\{i,n+1\}=\emptyset$ or $i=l$ then $\ini(x_i y_{n+1} -x_n y_i)$ and $\ini(f_{kl})$ form a regular sequence. Hence the corresponding $S$-pair reduces to $0$.
If $n+1\in \{k,l\}$, then $\sigma$ is an admissible path in $G_2$ and  $\ini(f_{kl})=x_{n+1}y_l$. Therefore in this case the initial monomials form a regular sequence, too.

It remains to consider the case $i=k$. We observe that there exists a monomial $w$ such that
\begin{equation}\label{eq:sp0}
 S(u_\pi(x_i y_{n+1} -x_n y_i), u_\sigma f_{il})=w (x_l y_{n+1}-x_n y_l),
\end{equation}
and
\begin{equation}\label{eq:sp1}
 S(u_\pi f_{in}, u_\sigma f_{il})= w f_{ln}
\end{equation}
Since \eqref{eq:sp1} reduces to $0$ in $\GG_1$, there exists $f\in \GG_1$ such that $\ini(f)$ divides $wx_l y_n$.

If $y_n$ divides $\ini(f)$ then $f=u_\tau f_{jn}$ and this implies that $f'=u_\tau(x_j y_{n+1}-x_n y_j)\in \GG_1$. Therefore
the remainder of $w f_{ln}$ with respect to $f$ is equal to the remainder of $w (x_l y_{n+1}-x_n y_l)$ with respect to $f'$ and reduce to $0$.

If $y_n$ does not divide $\ini(f)$ then $\ini(f)$ divides $w x_l$ and hence initial term of \eqref{eq:sp0}. That is the remainder of $w f_{ln}$ with respect to $f$ is 
\begin{equation}\label{eq:red2}
 w' f_{l'n}
\end{equation}
for some monomial $w'$, and at the same time the remainder of $w (x_l y_{n+1}-x_n y_l)$ with respect of $f$ is  
\begin{equation}\label{eq:red3}
 w' (x_{l'} y_{n+1}-x_n y_{l'}).
\end{equation}

By proceeding as before since $w' f_{l'n}$ is not zero and reduces to $0$ we can apply the same reduction step to $w' f_{l'n}$ and $w' (x_{l'} y_{n+1}-x_n y_{l'})$ following the arguments applied to the binomials in the second terms of equations \eqref{eq:sp0} and \eqref{eq:sp1}. Thanks to Buchberger's algorithm since \eqref{eq:red2} reduces to $0$ in a finite number of steps also \eqref{eq:red3} reduces to $0$ by the same number of steps.

Hence $\GG_1$ is a Gr\"obner basis and we can remove the reducible polynomials $u_\pi f_{ij}$ with $j=n$ since their initial terms are divisible by $\ini(l_y)=y_n$. The claim follows.

Therefore
\begin{equation}\label{eq:ini}
 \ini(J_{G'}+(l_y))=(y_n,u_\pi x_i y_j, u_\pi x_{i'} y_{n+1})\mbox{ with }i<j\neq n, i'<n.
\end{equation}

Suppose that $f\in(J_{G'}+l_y):l_x$, that is $f(x_n-x_{n+1})\in (J_{G'}+(l_y))$. This implies $\ini(f(x_n-x_{n+1}))=\ini(f)x_n\in \ini(J_{G'}+(l_y))$.
We observe that $x_n$ does not divide any monomial in the minimal set of generators of $\ini(J_{G'}+(l_y))$. In fact $i\neq n$ and $i'\neq n$ by \eqref{eq:ini}.
Let $\pi$ be an admissible path such that there exists $k$, with $1\leq k<r$ and $i_k=n$. Since $n$ is a free vertex in a clique $F\in \Delta(G')$, $\pi$ contains at least $2$ vertices $u$, $w\in F$ with $n\notin \{u,w\}$.
But since $\{u,w\}\in E(G)$ then $\pi$ is not admissible by condition (3). Hence $\ini(f)\in \ini(J_{G'}+(l_y))$. It easily follows that $f\in (J_{G'}+(l_y))$.
\end{proof}
Let $G=G_1\cup\cdots\cup G_r$ be a connected graph satisfying the following properties for all $i,j, k\in[r]$ which are pairwise different: 
\begin{enumerate}
  \item $|V(G_i)\cap V(G_j)|\leq 1$ and $V(G_i)\cap V(G_j)\cap V(G_k)=\emptyset$;
  \item If $V(G_i)\cap V(G_j)=\{v\}$ then $v$ is a free vertex in $\Delta(G_i)$ and $\Delta(G_j)$.
\end{enumerate}
  In order to characterize Cohen-Macaulay binomial edge ideals in this case, we associate with $G$  a graph $G^f$ whose  vertex set is $V(G^f)=\{1,\ldots, r\}$ and whose edge set is
 \[
 E(G^f)=\{\{i,j\}:V(G_i)\cap V(G_j)\neq\emptyset\}\}.
 \]
One can observe that $G^f$ is a connected graph since $G$ is a connected graph.
\begin{Corollary}\label{cor:fvi}
Let $G=G_1\cup\cdots\cup G_r$ be a connected graph satisfying properties $(1)$, $(2)$ and assume that the graph $G^f$ is a tree. If $S_i=K[\{x_j,y_j:j\in V(G_i)\}]$ for $i=1,\ldots,r$ then
\begin{equation}\label{eq:depthg}
\depth S/J_G=\depth S_1/J_{G_1}+\cdots +\depth  S_r/J_{G_r} -2(r-1).
\end{equation}
Moreover, $J_G$ is Cohen-Macaulay if and only if each $J_{G_i}$ is Cohen-Macaulay for $i=1,\ldots,r$.
\end{Corollary}
\begin{proof}
We use induction on $r$. We may assume that $r\geq 2$. Since $G^f$ is tree, there exists $G_i$ such that $v$ is a free vertex of $G'$ and $G_i$ where $G'=\bigcup_{j\neq i}G_j$. Moreover, $G'^f$ is again a tree. Applying the induction hypothesis and Theorem \ref{the:union}, the assertion follows.

\end{proof}
As a consequence of Corollary \ref{cor:fvi} we get the following

\begin{Corollary}\label{cor:cmequm}
Let $G=G_1\cup\ldots\cup G_r$ as above, and $G^f$ be a tree. Assume that $J_{G_i}$ is Cohen-Macaulay for $i=1,\ldots,r$. Then the following statements are equivalent:
\begin{enumerate}
  \item[(a)] $J_G$ is Cohen-Macaulay;
  \item[(b)] $J_G$ is unmixed.
\end{enumerate}
 \end{Corollary}
\begin{proof}
The implication (a) \implies (b) is well-known.

(b) \implies (a): Let $|V(G_i)|=m_i$ and $|V(G)|=n$. Since $G^f$ is a tree, a simple induction argument shows that  $n=\sum_{i=1}^r m_i-r+1$.

Now suppose that $S_i=K[\{x_j,y_j:j\in V(G_i)\}]$ for $i=1,\ldots,r$. Then $\dim S/J_G=n+1$, because $J_G$ is unmixed. Since $J_{G_i}$ is Cohen-Macaulay, $\dim S_i/J_{G_i}=\depth S_i/J_{G_i}=m_i+1$ for $i=1,\ldots,r$. It follows from Equation (\ref{eq:depthg}) that
\[
\depth S/J_G=\sum_{i=1}^r (m_i+1)-2(r-1)=(\sum_{i=1}^rm_i-r+1)+1=n+1.
\]
\end{proof}

\begin{Lemma}\label{lem:intree}
 Let $G$ be a chordal graph and $G_1,\ldots,G_r$ be its maximal cliques of $G$.  Suppose that $G$ satisfies condition (1) with respect to $G_1,\ldots,G_r$. Then
 $G^f$ is a tree.
\end{Lemma}
\begin{proof}
Suppose $G^f$ is not a tree. Then there exists a cycle $\{i_1,\ldots,i_s\}$ of length $s$ in $G^f$. We may assume that every proper subset of $\{i_1,\ldots,i_s\}$ is not a cycle. For each $i_j$, let $G_{i_j}\in \{G_1,\ldots,G_r\}$ be the corresponding clique in $G$. Then $V(G_{i_j})\cap V(G_{i_{j+1}})=\{v_{i_j}\}$ for $j=1,\ldots,s$ where $i_{s+1}=i_1$. Hence $C=\{v_{i_1},\ldots,v_{i_s}\}$ is a cycle of length $s$ in $G$ passing through all $G_{i_j}$ for $j=1,\ldots,s$. Note that $C$ is also chordal, since $G$ is chordal. Since every proper cycle of $C$ is not a cycle, it follows that $s=3$ and hence $C$ is a clique. This is a contradiction, since $|V(C)\cap V(G_{i_j})|>1$ for $j=1,2,3$.
\end{proof}
As a last result of this section we obtain as a special case Theorem 1.1 of \cite{HEH}.
\begin{Corollary}
  Let $G$ be a chordal graph and $G=G_1\cup\ldots\cup G_r$ such that $|V(G_i)\cap V(G_j)|\leq 1$ for $i\neq j\in\{1,\ldots,r\}$. Assume that $G_i$ is maximal clique for $i=1,\ldots,r$. Then the following statements are equivalent:
\begin{enumerate}
  \item[(a)] $J_G$ is Cohen-Macaulay;
  \item[(b)] $J_G$ is unmixed;
  \item[(c)] $V(G_i)\cap V(G_j)\cap V(G_k)=\emptyset$ for $i\neq j\neq k\in\{1,\ldots,r\}$.
\end{enumerate}
 \end{Corollary}
\begin{proof}
(a) \implies (b) is known. By Lemma \ref{lem:intree}, (c) together with hypothesis implies that $G^f$ is a tree. Hence (c) \implies (a) is a special case of Corollary \ref{cor:cmequm}.

(b) \implies (c): Suppose $V(G_i)\cap V(G_j)\cap V(G_k)\neq\emptyset$ for $i\neq j\neq k\in\{1,\ldots,r\}$. Then there is a vertex $v$ of $G$ which is intersection of at least three maximal cliques. It follows that $T=\{v\}$ has cut point property for $G$. Hence $c(T)\geq 3$, a contradiction.

\end{proof}

\section{Binomial edge ideals and cones}
The main goal of this section is to study the unmixed and Cohen-Macaulay  property of the binomial edge ideal of the cone on a graph.�

\begin{Lemma}\label{lemm:cone1conn}
Let $H$ be a connected graph, and let
\[
G=\cone(v,H).
\]
Then
\[
\begin{array}{rcl}
 \mathcal C(G)&=&\{T\subset V(G):T=T'\cup\{v\}  \text{ with $T'\neq \emptyset$ and  $T'\in \mathcal C (H)$}\} \cup \{\emptyset\}.
\end{array}
 \]
Moreover, $\height P_T=\height P_{T'}+2$, for all $T\neq \emptyset.$
\end{Lemma}
\begin{proof}
We assume $T=T'\cup\{v\}$ where $T'\in \mathcal C (H)$ and $T'\neq \emptyset$ and we want to show that $T\in \mathcal C(G)$. Note that $G_{\ol{T}}=H_{\ol{T'}}$. Hence for all $i\in T'$, $i$ is a cut point of $H_{\ol{T'}\cup\{i\}}=G_{\ol{T}\cup\{i\}}$. Let $i=v$ then $c(T\setminus \{v\})=1$ since $G$ is a cone defined on $v$. Hence $v$ is a cut point on the induced subgraph $G_{\ol{T}\cup\{v\}}$ since $c(T)>1$.

The other inclusion can be proved by similar arguments,  observing that
\begin{enumerate}
 \item since $H$ is connected, if $T=\{v\}$ then  $c(T)=c(T\backslash\{v\})=1$ and 
 \item if $v\not\in T\neq \emptyset$ then $c(T)=c(T\backslash \{i\})=1$ for all $i\in T$,
\end{enumerate}
hence $T$ does not have cut point property in both cases.

Consider $T\in\mathcal C(G)$ with $|T|>0$. It follows that $T=T'\cup\{v\}$ with $|T'|>0$. Since $G_{\ol{T}}=H_{\ol{T'}}$, we have $\height P_T=n+|T|-c(T)=n+|T'|+1-c(T')=\height P_{T'}+2$.
\end{proof}
\begin{Corollary}
Let $H$ be a connected graph, and let
\[
G=\cone(v,H)
\]
with $|V(G)|=n$.
Then
\[
\dim S/J_G= \max\{n+1,\dim S'/J_H\}
\]
where $S=K[\{x_i,y_i:i\in V(G)\}]$ and $S'=K[\{x_i,y_i:i\in V(H)\}]$.
\end{Corollary}

 \begin{Theorem}\label{the:cone1conn}
 Let $H$ be a connected graph and assume that  $J_H$ is unmixed. Let
 $G=\cone(v,H)$.
 Then the following conditions are equivalent:
 \begin{enumerate}
  \item[(a)] $H$ is a complete graph;
  \item[(b)] $J_G$ is unmixed.
 \end{enumerate}
 If the equivalent conditions hold, then $J_G$ is Cohen-Macaulay.
 \end{Theorem}
 \begin{proof}
 (a) \implies  (b): If $H$ is a complete graph, then $G$ is also a complete graph, and hence $J_G$ is a prime ideal.

 (b) \implies (a): Let $|V(G)|=n$ and assume $H$ is not a complete graph. Then there exists $T'\in\mathcal C (H)$ with $T'\neq \emptyset$. By Lemma \ref{lem:unmcomp} the number of connected components of $H_{\ol{T'}}$ are $|T'|+1$ and $\height P_{T'}=n-2$. By Lemma \ref{lemm:cone1conn}, $P_T$ with $T=T'\cup\{v\}$ is a minimal prime ideal of $J_G$ and $\height P_T=\height P_{T'}+2=n$. Since $\height P_\emptyset(G)=n-1$ we obtain that $J_G$ is not unmixed.
 \end{proof}

If $J_G$ is unmixed then it does not implies that $J_H$ is unmixed. In \cite{HEH}, the authors give the following example (see Figure \ref{fan}) of the unmixed graph which is not Cohen-Macaulay and also $J_H$ is not unmixed by Lemma \ref{lem:unmcomp}.
\begin{figure}[hbt]
\begin{center}
\psset{unit=0.8cm}
\begin{pspicture}(0,0)(10,4)
\put(0.5,1){$H$}
\put(3,3.5){$v$}
\rput(1.5,0.5){$\bullet$}
\rput(2.5,1.5){$\bullet$}
\rput(1,2.5){$\bullet$}
\rput(4,2.5){$\bullet$}
\rput(2.5,3.5){$\bullet$}
\psline(1.5,0.5)(2.5,1.5)
\psline(2.5,1.5)(1,2.5)
\psline(2.5,1.5)(4,2.5)

\rput(6.5,0.5){$\bullet$}
\rput(7.5,1.5){$\bullet$}
\rput(6,2.5){$\bullet$}
\rput(9,2.5){$\bullet$}
\rput(7.5,3.5){$\bullet$}
\psline(6.5,0.5)(7.5,1.5)
\psline(7.5,1.5)(6,2.5)
\psline(7.5,1.5)(9,2.5)
\psline(7.5,3.5)(9,2.5)
\uput{0.2}[0](5.5,1){$G$}
\psline(7.5,3.5)(7.5,1.5)
\psline(7.5,3.5)(6,2.5)
\psline(7.5,3.5)(6.5,0.5)
\end{pspicture}
\end{center}
\caption{}\label{fan}
\end{figure}
\begin{Lemma}\label{cor:only2}
Let $H=\bigsqcup_{i=1}^r H_i$ be a graph with $H_i$ connected components with $r\geq 1$ and let
\[
G=\cone(v,H).
\]
If $J_G$ is unmixed then $H$ has at most two connected components.
\end{Lemma}
\begin{proof}
  Suppose $r\geq 3$. Then $v$ is a cut point of $G$. Let $T=\{v\}$. Since $T\in \mathcal C(G)$ by Lemma \ref{lem:unmcomp} it follows that $J_G$ is not unmixed.
\end{proof}

\begin{Lemma}\label{lemm:cone2conn}
Let $H=H_1\sqcup H_2$ such that $H_1$ and $H_2$ are connected graphs and let
\[
G=\cone(v,H).
\]
Then

\[
\mathcal C(G)=\{T\subset V(G):T=T_1\cup T_2\cup\{v\}, T_i\in \mathcal C (H_i) \text{ for } i=1,2\}
 \cup \{\emptyset\}.
 \]
Moreover, $\height P_T=\height P_{T_1}+\height P_{T_2}+2$, for all $T\neq \emptyset.$
\end{Lemma}
\begin{Corollary}\label{cor:dimn1}
Let $H=H_1\sqcup H_2$ such that $H_1$ and $H_2$ are connected graphs and let
\[
G=\cone(v,H).
\]
Then
\[
\dim S/J_G=\max\{\dim S_1/J_{H_1}+\dim S_2/J_{H_2},n+1\}
\]
where $S_i=K[\{x_j,y_j:j\in V(H_i)\}]$ for $i=1,2$ and $S=K[\{x_k,y_k:k\in V(G)\}]$.
\end{Corollary}

\begin{Corollary}\label{cor:unmixed}\label{cor:dimh2}
Let $H=H_1\sqcup H_2$ such that $H_1$ and $H_2$ are connected graphs and let
\[
G=\cone(v,H).
\]
The following conditions are equivalent:
\begin{enumerate}
 \item[(a)] $J_{H_1}$ and $J_{H_2}$ are unmixed;
 \item[(b)] $J_G$ is unmixed.
\end{enumerate}
\end{Corollary}
\begin{proof}

We may focus our attention on $T\in \mathcal C (G)$ with $T\neq \emptyset$. By Lemma \ref{lemm:cone2conn},  $T=\{v\}\cup T_1\cup T_2$ where $T_1\in \mathcal C (H_1)$, $T_2\in \mathcal C (H_2)$. We observe that
\begin{equation}\label{eq:conn}
 c(T)=c(T_1)+c(T_2)
\end{equation}
where $c(T_i)$ is the number of connected components of  $(H_i)_{\ol{T}_i}$ with $i=1,2$.

(a) \implies (b):  It follows from Lemma \ref{lem:unmcomp} that $c(T_i)=|T_i|+1$ for $i=1,2$. By putting these values in Equation \eqref{eq:conn}, we get the result by Lemma \ref{lem:unmcomp}.

(b) \implies (a):  Suppose $J_{H_1}$ is not unmixed. Then there exists $T_1\in\mathcal C(H_1)$ such that $c(T_1)\neq |T_1|+1$. Let $T=\{v\}\cup T_1$. Lemma \ref{lemm:cone2conn} implies that $T\in\mathcal C(G)$. Since $c(T_2=\emptyset)=1$ we obtain a contradiction by \eqref{eq:conn}.
\end{proof}
\begin{Theorem}\label{theo:cone2CM}
Let $H=H_1\sqcup H_2$ such that $H_1$ and $H_2$ are connected graphs and  let
\[
G=\cone(v,H).
\]
If $J_{H_1}$ and $J_{H_2}$ are Cohen-Macaulay then $J_G$ is Cohen-Macaulay.
\end{Theorem}

\begin{proof}
Since $H=H_1\sqcup H_2$, $J_{H_1}$ and $J_{H_2}$ are Cohen-Macaulay if and only if $J_{H}$ is Cohen-Macaulay. Let $|V(G)|=n$. Let $T\in \mathcal C(G)$. By Lemma \ref{lemm:cone2conn} we have two cases:
\begin{enumerate}
 \item $T=\emptyset$;
 \item $T=\{v\}\cup T_1\cup T_2$ with $T_1\in \mathcal C(H_1)$, $T_2\in \mathcal C(H_2)$.
\end{enumerate}
Hence we have $J_G=Q_1\cap Q_2,$ where
\begin{eqnarray}
 Q_1 & = & \bigcap_{T\in \mathcal C(G),\,v\not\in T}P_T(G)=P_{\emptyset}(G),\\
 Q_2 & = & \bigcap_{T\in \mathcal C(G),\,v\in T}P_T(G).
\end{eqnarray}
Since $P_{\emptyset}(G)$ is the ideal generated by all $2\times 2$-minors, we obtain that $\depth S/Q_1=n+1$. With respect to $Q_2$, we observe that for all $T\subset V(G)$ with $v\in T$, we have $P_T(G)=(x_v,y_v)+P_{T\setminus \{v\}}(H)$. It follows that $Q_2=(x_v,y_v)+J_{H}.$ Let $S_1$ be the polynomial ring $S/(x_v,y_v)$. Then $S/Q_2\simeq S_1/J_{H}$ where $J_H$ is a Cohen-Macaulay binomial edge ideal with $2$ connected components. Hence $\depth S_1/J_H=(n-1)+2=n+1$. We also observe that $Q_1+Q_2=P_{\emptyset}(G)+(x_v,y_v)+J_H=P_{\emptyset}(G)+(x_v,y_v)$. Thus $S/(Q_1+Q_2)\simeq S_1/J_{\wt{H}}$, where $\wt{H}$ is the complete graph on $n-1$ vertices. Hence $\depth S/(Q_1+Q_2)=n$. By the following exact sequence
\[
0\longrightarrow S/J_{G}\longrightarrow S/Q_1\oplus S/Q_2\longrightarrow S/Q_1+Q_2\longrightarrow 0.
\]
and depth lemma we obtain that $\depth S/J_{G}= n+1$. By Corollary \ref{cor:dimn1}, we get $\dim S/J_G=n+1$.
\end{proof}

\begin{Example}
 Let $H=H_1\sqcup H_2$ be a graph where $H_1$ and $H_2$ are path graphs of length two. Let $v$ be an isolated vertex. By Theorem \ref{theo:cone2CM} the graph $G=cone(v,H)$ is a Cohen-Macaulay (see Figure \ref{cone}).

\begin{figure}[hbt]
\begin{center}

\psset{unit=0.8cm}

\begin{pspicture}(0,0)(11,5)

\rput(1,2){$\bullet$}
\rput(2,1){$\bullet$}
\rput(2,3){$\bullet$}
\rput(3,2){$\bullet$}
\rput(4,1){$\bullet$}
\rput(4,3){$\bullet$}
\rput(5,2){$\bullet$}

\put(1,.7){$H_1$}
\put(2.8,1.4){$v$}
\put(4.4,.7){$H_2$}

\psline(1,2)(2,3)
\psline(1,2)(2,1)

\psline(4,1)(5,2)

\psline(4,3)(5,2)

\put(5.8,1.8){$\Rightarrow$}

\rput(7,2){$\bullet$}
\rput(8,1){$\bullet$}
\rput(8,3){$\bullet$}
\rput(9,2){$\bullet$}
\rput(10,1){$\bullet$}
\rput(10,3){$\bullet$}
\rput(11,2){$\bullet$}

\psline(7,2)(8,3)
\psline(7,2)(9,2)
\psline(7,2)(8,1)
\psline(8,1)(9,2)
\psline(8,3)(9,2)

\psline(9,2)(10,3)

\psline(9,2)(11,2)
\psline(9,2)(10,1)
\psline(10,1)(11,2)

\psline(10,3)(11,2)

\put(8.7,.7){$G$}

\end{pspicture}
 
\end{center}
\caption{}\label{cone} 
\end{figure}

\end{Example}
\bibliographystyle{plain}

\begin{thebibliography}{1}




\bibitem{Co}  CoCoATeam,
 {\em CoCoA: a system for doing 
     Computations in Commutative Algebra},
 Available at http://cocoa.dima.unige.it


\bibitem{CR}{M. Crupi, G. Rinaldo}, \textit{Binomial edge ideals with quadratic Gr\"obner bases},
The Electronic Journal of Combinatorics, ISSN 1077-8926, 18, 2011, pp. 1--13.

\bibitem{DES}{P. Diaconis, D. Eisenbud, B. Sturmfels},
\textit{Lattice walks and primary decomposition},
\newblock Mathematical Essays in honor of Giancarlo Rota, Birh\"{a}user, Boston, Cambridge, MA, 1998, 173--193.


\bibitem{Ha} {F. Harary}, \textit{Graph theory},  \newblock
Addison-Wesley series in Mathematics, 1972.

\bibitem{HH1} {J. Herzog, T. Hibi}, \textit{ Distributive lattices, bipartite graphs and Alexander duality},  \newblock
J. Alg. Combin. {\bf 22} (2005) 289--302.

\bibitem{HH}{J. Herzog, T. Hibi, F. Hreinsdottir, T. Kahle, J. Rauh},
\textit{Binomial edge ideals and conditional independence statements},
\newblock Advances in Applied Mathematics, {\bf 45} (2010) 317--333.

\bibitem{HEH}{V. Ene, J. Herzog, T. Hibi}, \textit{Cohen-Macaulay binomial edge ideals}, Nagoya Math. J. Volume 204 (2011), 57--68.

\bibitem{Mk}Brendan D. McKay,
{\em NAUTY: No AUTomorphisms, Yes?}, 
Available at http://cs.anu.edu.au/{\textasciitilde}bdm/nauty/

\bibitem{MO}{M. Ohtani},
\textit{Graphs and ideals generated by some $2$-minors},  Comm. Algebra 39 (2011), 905--917.

\bibitem{SVV} {A. Simis, W. Vasconcelos,  R.H. Villarreal}, \textit{On the ideal theory of graphs},
   \newblock  J. Algebra {\bf 167} (1994) No.2 389--416.

\bibitem{Vi2} R.H. Villarreal,
   \textit{Cohen-Macaulay graphs}, Manuscripta Math. {\bf 66} (1990) 277--293.
\end{thebibliography}

\end{document}